\documentclass[10pt]{amsart}
\usepackage{amsmath,amssymb,amsfonts,amsthm,amsopn}
\usepackage{latexsym}
\usepackage{mathrsfs}
\usepackage{geometry}
\usepackage{xcolor}
\usepackage{color, colortbl}
\usepackage{pb-diagram}
\usepackage[title]{appendix}
\usepackage{tikz-cd}
\usepackage{tikz}
\usepackage{appendix}
\usepackage{mathtools}
\usepackage{booktabs}
\usepackage{enumerate}
\usepackage{hyperref}
\usepackage{multirow}
\usetikzlibrary{arrows.meta}
\usetikzlibrary{angles, quotes, arrows.meta, patterns, decorations.pathreplacing, calc, intersections}

\mathtoolsset{showonlyrefs}

\usepackage{multirow}

\newtheorem{theorem}{Theorem}[section]
\newtheorem{lemma}[theorem]{Lemma}
\newtheorem{corollary}[theorem]{Corollary}
\newtheorem{proposition}[theorem]{Proposition}
\newtheorem{definition}[theorem]{Definition}

\newtheorem{remark}[theorem]{Remark}
\theoremstyle{definition}

\newcommand{\beqa}{\begin{eqnarray*}}
	\newcommand{\eeqa}{\end{eqnarray*}}
\newcommand{\HH}{\mathfrak H}

\DeclareMathOperator*{\Sp}{Sp}
\DeclareMathOperator*{\Mp}{Mp}
\DeclareMathOperator*{\Sym}{Sym}
\DeclareMathOperator*{\Op}{{Op}}

\DeclareMathOperator*{\GL}{GL}

\newcommand{\field}[1]{\mathbb{#1}}
\newcommand{\bR}{\field{R}}        %  real numbers
\newcommand{\bN}{\field{N}}        %  natural numbers
\newcommand{\bC}{\field{C}} 
\newcommand{\C}{\field{C}} %  complex numbers
\newcommand{\bF}{\field{F}} 
\newcommand{\cL}{\mathcal{L}}     %  
\def\Fur{\mathscr{F}}              % Calligraphic Letters
\def\cS{\mathcal{S}}
\def\cD{\mathcal{D}}

\def\cB{\mathcal{B}}

\def\cM{\mathcal{M}}

\def\rd{\bR^d}

\def\rdd{{\bR^{2d}}}

\def\R{\right)}

\def\<{\left<}
\def\>{\right>}

\def\ud{\, \mathrm{d}}

\def\mv1{M_v^1}

\def\Shilov{S_{1/2}^{1/2}(\mathbb{R}^d)}
\def\Shilovh{S_{1/2,h}^{1/2}(\mathbb{R}^d)}
\def\Shilovp{S_{1/2}^{1/2}(\mathbb{R}^d)'}

\def\mn{(m,n)}
\def\mn'{(m',n')}

\newcommand{\norm}[1]{\lVert#1\rVert}

\def\w{\mathrm{w}}

\def\N{\mathbb{N}}
\def\R{\mathbb{R}}
\def\Ren{\mathbb{R}^d}

\def\Fur{\mathcal{F}}

\def\f{\varphi}

\def\Sn2{S_{2}(L^{2}(\Ren))}
\def\S1{S_{1}(L^{2}(\Ren))}
\def\sig00{\sigma_{0,0}}

\newcommand{\op}{\mathrm{op}}

\begin{document}

\begin{abstract}
    The oscillator (or metaplectic) semigroup is classically defined as the two-fold cover of the semigroup of positive complex symplectic matrices.
    Although geometrically precise, this definition does not provide an intrinsic operator-theoretic characterization of the bounded operators corresponding to positive symplectic matrices.
    This is in contrast with the real metaplectic group, whose relation with the symplectic group can be expressed directly in terms of the Schr\"odinger representation of the Heisenberg group through its intertwining property.
    In the complex setting, such a relation appears in the existing literature mainly in infinitesimal form or on suitable classes of Gaussian functions.
    In this work we prove that the (complexified) Schr\"odinger intertwining relation holds for every function in $L^2(\R^d)$.
    The main point is the converse statement: if an arbitrary complex symplectic matrix $S$ admits a nonzero bounded operator on $L^2(\R^d)$ satisfying this intertwining relation, then $S$ is necessarily positive and the operator coincides, up to a nonzero scalar, with the corresponding element of the oscillator semigroup.
    Consequently, positive complex symplectic matrices are exactly those admitting a nonzero bounded Schr\"odinger intertwiner, and the associated intertwining space is one-dimensional.
    This provides a Schr\"odinger-representation characterization of the oscillator semigroup, parallel to the classical one for the real metaplectic group.
\end{abstract}

\title[A Schr\"odinger characterization of the oscillator semigroup]{A Schr\"odinger characterization of the oscillator semigroup}

\author{Gianluca Giacchi}
\address{Universit\`a della Svizzera Italiana, Via la Santa 1, 6962 Lugano, Switzerland}
\email{gianluca.giacchi@usi.it}

\maketitle

\section{Introduction}
The purpose of this paper is to reintroduce the {\em oscillator semigroup} as the semigroup of intertwiners for the complexified Schrödinger representation of the Heisenberg group, thereby recovering the classical framework in which metaplectic operators arise as intertwiners of the Schrödinger representation.
As a by-product, the {\em semigroup of positive symplectic matrices} is identified as the subset of complex symplectic matrices for which such a relation can hold for some operator bounded on $L^2$.

\subsection{The metaplectic group and the Schr\"odinger representation}
The metaplectic group $\Mp(d,\bR)$ is the two-fold cover of the real symplectic group $\Sp(d,\bR)$.
Recall that a matrix $U\in\bR^{2d\times2d}$ is symplectic if $U^\top JU=J$, where
\begin{equation}\label{intro.defJ}
    J=\begin{pmatrix} O & I\\
        -I & O
    \end{pmatrix}
\end{equation}
represents the canonical symplectic form of $\rdd$, here denoted as $\sigma$, i.e.,
\begin{equation}\label{intro.defSigma}
    \sigma(z,w)=Jz\cdot w, \qquad z,w\in\rdd.
\end{equation}
The Schr\"odinger representation of the Heisenberg group is the unitary irreducible representation $\rho:\rdd\times\bR\to\mathcal{U}(L^2(\rd))$ defined by
\begin{equation}\label{intro.Schrodingerrep}
    \rho(x,\xi;\tau)f(t)=e^{2\pi i\tau}e^{-i\pi x\cdot\xi}e^{2\pi i\xi\cdot t}f(t-x), \qquad x,\xi\in\rd,\;\tau\in\bR,\;f\in L^2(\rd).
\end{equation}
The intertwiners of this representation are called {\em metaplectic operators}: for every $U\in\Sp(d,\bR)$ there exists an operator $\widehat U\in\mathcal{U}(L^2(\rd))$, unique up to a phase factor (complex constant of modulus 1) such that
\begin{equation}\label{intro.intertU}
    \widehat U\rho(z;\tau)\widehat U^{-1}=\rho(Uz;\tau), \qquad z\in\rdd,\;\tau\in\bR.
\end{equation}
The group $\{c\widehat U:|c|=1,\, U\in\Sp(d,\bR)\}$ has a subgroup that realizes the two-fold cover of $\Sp(d,\bR)$, called the {\em metaplectic group} and here denoted by $\Mp(d,\bR)$.
Therefore, in $\Mp(d,\bR)$ there are precisely two operators for each real symplectic matrix.
Specifically, if $\widehat U\in\Mp(d,\bR)$, then also $-\widehat U\in\Mp(d,\bR)$.
{The metaplectic/oscillator representation and its intertwining interpretation go back to the classical works of Shale and Weil \cite{Shale1962,Weil1964}; see also Howe \cite{Howe1980} and \cite{Folland} for a systematic treatment.}

\subsection{From the real to the complex framework}
There are two main reasons to extend the construction of the metaplectic group to the complex symplectic group $\Sp(d,\bC)$ {\em consistently}.
Recall that a matrix $S\in\bC^{2d\times2d}$ is symplectic if $S^\top JS=J$, exactly as in the real case.

The first reason is practical: the double cover of $\Sp(d,\bR)$ is a group homomorphism, and so when it comes to $\Mp(d,\bR)$, operator calculus traces back to linear algebra, enabling high geometric interpretability of metaplectic operators in terms of the structure of their projection.
In principle, extending the metaplectic representation to complex matrices, means extending these advantages to a wider class of operators.

The second reason is analytical: among metaplectic operators we find the evolution operators of Schr\"odinger Cauchy problems with real quadratic Hamiltonians:
\begin{equation}\label{intro.Schro}
        i\frac{1}{2\pi}\partial_tu={\Op}^\w(a)u,\qquad u(0,\cdot)=u_0\in\cS(\rd),
\end{equation}
where
\begin{equation}
    {\Op}^\w(a)f(x)=\int_{\rdd}f(y)a\big(\frac{x+y}{2},\xi\big)e^{2\pi i(x-y)\cdot\xi}\ud y\ud\xi,\qquad f\in\cS(\rd),
\end{equation}
is the Weyl quantization of the quadratic form $a(z)$, $z\in\rdd$.
Specifically, if $S_t$ is the Hamiltonian flow associated with \eqref{intro.Schro}, the evolution operator $u(t,\cdot)=e^{-2\pi it\Op^\w(a)}u_0$ is metaplectic and has projection $S_t$, i.e., $e^{-2\pi it\Op^\w(a)}=\widehat S_t$, up to a phase.
When considering {\em complex} quadratic Hamiltonians, that is when $a$ in \eqref{intro.Schro} is a complex quadratic form with positive semidefinite imaginary part, the Hamiltonian flow becomes a complex symplectic matrix, and similar considerations are still possible: by extending the metaplectic representation to complex symplectic matrices consistently, one obtains precisely that the Hamiltonian flow is still the projection of the evolution operator.
{This is also what makes such an extension relevant for time-frequency analysis, where metaplectic operators are used systematically to describe function spaces, phase-space representations and propagators of Schr\"odinger equations \cite{CorderoGiacchi2023,CorderoGiacchiRodino2024,NicolaTrapasso}.}

However, there are obstructions in such a consistent extension.
The operator associated with
\begin{equation}\label{intro.defVPT}
    V_P^\top=\begin{pmatrix} I & P\\
        O & I \end{pmatrix},
\end{equation}
when $P$ is a real symmetric matrix is, up to a phase, the Fourier multiplier $\mathfrak m_{-P}f=\Fur^{-1}\big(\Phi_{-P}\widehat f\big)$, where $\Phi_{-P}(x)=e^{-i\pi Px\cdot x}$.
Therefore, a consistent generalization of the metaplectic representation shall associate to $V_P^\top$ in \eqref{intro.defVPT} with $P$ complex and symmetric, the corresponding operator $\mathfrak m_{-P}$, which is defined on $L^2(\rd)$ if and only if $\Im(P)\leq0$.
Another example is provided by the operator associated with
\begin{equation}\label{intro.defDE}
    \cD_E=\begin{pmatrix} E^{-1} & O\\
        O & E^\top \end{pmatrix},
\end{equation}
when $E\in\GL(d,\bR)$.
This is, up to a phase, the unitary rescaling $\mathfrak T_Ef(x)=|\det(E)|^{1/2}f(Ex)$.
This extends consistently to $E\in\GL(d,\bC)$ in terms of pseudodifferential calculus, as
\begin{equation}
    \mathfrak T_Ef(x)=|\det(E)|^{1/2}\int_{\rd}\widehat f(\xi)e^{2\pi iEx\cdot\xi}\ud\xi,
\end{equation}
but this operator is not even defined on $\cS'(\rd)$.

It seems clear that there must be an admissibility criterion for complex symplectic matrices that excludes these pathological cases.

\begin{definition}
    A matrix $S\in\Sp(d,\bC)$ is {\em positive} if the Hermitian form
        \begin{equation}\label{intro.defPos}
            H_S=i^{-1}\big(S^\ast JS-J\big)
        \end{equation}
    is positive semidefinite.
\end{definition}
This defines a semigroup of complex symplectic matrices and it turns out that the real symplectic group is its group of units, $\Sp_+(d,\bC)^\times=\Sp(d,\bR)$.
Similarly to the {\em real} case, this semigroup admits a two-fold cover (as a semigroup) consisting of $L^2$-contractions.
\begin{definition}
    The {\em oscillator} (or metaplectic) {\em semigroup} is the two-fold cover of $\Sp_+(d,\bC)$.
    It is denoted here by $\Mp_+(d,\bC)$.
\end{definition}
If $S\in\Sp_+(d,\bC)$, we denote by $\widehat S$ any corresponding operator in $\Mp_+(d,\bC)$.

\subsection{A time jump}
The semigroup $\Sp_+(d,\bC)$ and its metaplectic realization $\Mp_+(d,\bC)$ have appeared in several equivalent forms over the last century.

Early constructions were given in the Bargmann-Fock space by Kramer, Moshinsky and Seligman \cite{KMS1975} and, subsequently, by Brunet and Kramer \cite{BrunetKramer1980}.
Boundedness of the operators associated with complex symplectic matrices naturally leads there from $\Sp(d,\bR)$ to $\Sp_+(d,\bC)$.
Brunet later developed the corresponding \emph{metaplectic semigroup}, including its covering and contractive operator structure \cite{Brunet1985}.
{The underlying Hilbert space of entire functions and the associated integral transform were introduced by Bargmann \cite{Bargmann1961}, and the action of positive linear transformations on such spaces has been revisited in \cite{Sengupta}.}

In a different language, Ol'shanskii's theory of invariant cones and holomorphic Lie semigroups realizes the same general phenomenon as the holomorphic continuation of unitary representations to semigroups of contractions \cite{Olshanskii1982}.
In particular, for Hermitian groups he characterizes the maximal complex semigroup geometrically by preservation of the associated bounded symmetric domain.
{For a systematic account of Lie semigroups and of their holomorphic representations, we address the monograph by Hilgert and Neeb \cite{HilgertNeeb}.}

Howe's oscillator semigroup \cite{Howe1988} gives the corresponding Gaussian integral-operator realization on $L^2(\R^d)$, an approach proved to be equivalent to Brunet's metaplectic semigroup by Hilgert \cite{Hilgert1989}.

Finally, H\"ormander's positive complex canonical transformations and Gaussian Fourier integral operators provide another closely related formulation \cite{Hormander1995}.
{This formulation relies in turn on the calculus of Fourier integral operators with complex-valued phase functions of Melin and Sj\"ostrand \cite{MelinSjostrand} and on the corresponding $L^2$ estimates \cite{Hormander1983}.}

{More recently, Derezi\'nski and Karczmarczyk have described the oscillator semigroup through the Weyl calculus, as the set of operators of the form $c\cdot\Op^\w(e^{-A})$ with $A$ a quadratic form with positive definite real part; they relate this Gaussian/Weyl realization to the corresponding positive complex symplectic semigroup and compute explicitly operator and trace norms of its elements \cite{DerezinskiKarczmarczyk}.}

Despite their rather different starting points, these works share the common feature that the oscillator semigroup is presented as a two-fold cover of $\Sp_+(d,\bC)$.
In other words, the characterization of $\Mp(d,\bR)$ through the Schr\"odinger representation is not taken as the defining principle of its complex extension.

Nevertheless, the idea that the oscillator semigroup and the Schr\"odinger representation of the Heisenberg group are related is already present in the aforementioned pioneering contributions.
More precisely, let
\begin{equation}\label{intro.blockS}
    S=\begin{pmatrix} A & B\\
        C & D \end{pmatrix}\in\Sp(d,\bC), \qquad A,B,C,D\in\bC^{d\times d},
\end{equation}
and write $(y',\eta')=S(y,\eta)$.
In the notation of the present paper, the infinitesimal intertwining relation associated with $S$ takes the explicit form
\begin{equation}\label{intro.infinitesimal-intertwining}
    \widehat S\bigl( -y\cdot\nabla+2\pi i\eta\cdot t \bigr)f=\bigl(-y'\cdot\nabla+2\pi i\eta'\cdot t\bigr)\widehat S f, \qquad (y,\eta)\in\bC^{2d}.
\end{equation}
Thus the first-order position-momentum operator associated with $(y,\eta)$ is intertwined with the one associated with $(y',\eta')=S(y,\eta)$.
Relation \eqref{intro.infinitesimal-intertwining} is precisely the infinitesimal counterpart of \eqref{intro.intertU}, when $\widehat S\in\Mp(d,\bR)$.

A relation of this type is already contained in Howe's construction of the oscillator semigroup \cite[Section~19]{Howe1988}.
For the real metaplectic group, Howe first establishes the finite conjugation property on the twisted-convolution algebra.
He then passes to the corresponding infinitesimal identities and extends them holomorphically to the oscillator semigroup \cite{Howe1988}.
After translating the twisted differentiations through the Weyl transform and matching the normalization with \eqref{intro.Schrodingerrep}, Howe's relation becomes \eqref{intro.infinitesimal-intertwining}, initially on Schwartz space.
What is extended to $\Mp_+(d,\bC)$ in Howe's work is therefore the infinitesimal relation \eqref{intro.infinitesimal-intertwining}, rather than \eqref{intro.intertU}.

H\"ormander obtains essentially the same intertwining relation on $\cS(\rd)$ directly from the Gaussian kernel of the operator associated with a matrix in $\Sp_+(d,\bC)$ \cite[Proof of Proposition 5.8]{Hormander1995}.

The earlier Bargmann-space approach of Kramer, Moshinsky and Seligman contains the same phenomenon.
Their bounded operators associated with complex symplectic matrices satisfy the analytically continued first-order canonical equations \cite[Equation (5.52)]{KMS1975}.
Transporting these equations back from the Bargmann realization to the Schr\"odinger realization gives, in dimension one, precisely the two relations
\begin{align}
    \widehat S(-\partial_t)=(-a\partial_t+2\pi ict)\widehat S, \qquad \widehat S(2\pi it)=(-b\partial_t+2\pi id t)\widehat S, \qquad S=\begin{pmatrix} a & b\\
        c & d \end{pmatrix}\in{\Sp}_+(1,\bC),
\end{align}
which are the one-dimensional components of \eqref{intro.infinitesimal-intertwining}.
Importantly, they explicitly observe that, for the complex semigroup, these differential relations cannot in general be rewritten as a conjugation identity involving the inverse of $\widehat S$, since that inverse need not be bounded.

More recently, Viola proved that
\begin{equation}
    \widehat S{\Op}^\w(a)={\Op}^\w(a\circ S^{-1})\widehat S,
\end{equation}
holds for polynomial symbols and strictly positive complex canonical transformations, i.e., $S\in\Sp_+(d,\bC)$ with $H_S>0$.
For linear $a$, this is precisely the infinitesimal Schr\"odinger intertwining relation above.
The complexification of $\rho$ is instead considered on functions with Gaussian decay, since in general the complexified operators do not map $\mathcal S(\R^d)$ to $\mathcal S'(\R^d)$ \cite{Viola2017}.

{The problem of deciding which complex symplectic matrices carry a bounded operator is intimately related to the analysis of the solution operators of evolution equations generated by quadratic operators, for which boundedness, compactness and norm estimates have been obtained by Aleman and Viola \cite{AlemanViola} and Viola \cite{ViolaNorm}, while related subelliptic and Gelfand-Shilov regularization properties of quadratic semigroups were studied by Hitrik, Pravda-Starov and Viola \cite{HitrikPravdaStarovViola}.
Theorem \ref{thm:main} below may be regarded as the representation-theoretic counterpart of such boundedness criteria: it replaces the explicit analysis of the operator by a single intertwining property.}

Finally, the intertwining relation for the oscillator semigroup is proved in \cite{GRT} in the following form.
\begin{proposition}\label{intro.propGRT}
    Let $\widehat S\in\Mp_+(d,\bC)$ have projection \eqref{intro.blockS}.
    Then,
    \begin{equation}\label{intro.intertGauss}
        \widehat S\rho(z;\tau)f=\rho(Sz;\tau)\widehat Sf, \qquad z\in\rdd,\;\tau\in\bR
    \end{equation}
    for every function of the form
    \begin{equation}
        f(t)=\sum_{j=1}^N c_j\rho(w_j;\tau_j)e^{-\pi M_jt\cdot t}, \qquad N\geq 1,\; c_j\in\bC,\; w_j\in\rdd,\; \tau_j\in\bR,\; M_j>0.
    \end{equation}
    In particular, the intertwining relation holds densely on $L^2(\rd)$.
\end{proposition}

\subsection{Present contributions}
The perspective adopted in this work is reversed: we focus on the validity of
\begin{equation}\label{intro.finite-intertwining}
    T\rho(z;\tau)f=\rho(Sz;\tau)Tf, \qquad z\in\R^{2d},\; f\in L^2(\rd),
\end{equation}
where $T\in\mathcal{B}(L^2(\rd))$ and $\rho(Sz;\tau)$ is defined, densely on $L^2$, by complexifying \eqref{intro.Schrodingerrep}.

For complex $S$, this is substantially stronger as a formulation than the infinitesimal identities above: even though $z$ is real, $Sz$ is in general complex, and hence the expression $\rho(Sz;\tau)Tf$ is not a priori an $L^2$-function.
Comparing to Proposition \ref{intro.propGRT}, here we realize the complex Schr\"odinger action in $\Shilovp$, the topological dual of the Gelfand-Shilov class $\Shilov$, which makes \eqref{intro.finite-intertwining} meaningful for every $f\in L^2(\rd)$.

More importantly, we reverse the direction of the classical constructions: throughout this paper, $S$ is not assumed to be positive and $T$ is not assumed to belong to the oscillator semigroup.
We show that the mere existence of a nonzero bounded operator satisfying \eqref{intro.finite-intertwining} already forces $S\in\Sp_+(d,\bC)$ and determines $T$, up to a nonzero scalar, as the corresponding operator in $\Mp_+(d,\bC)$.
Our main result can be synthesized as follows.
\begin{theorem}\label{thm:main}
Let $S\in\Sp(d,\bC)$.
The following statements are equivalent.
\begin{enumerate}[(i)]
    \item $S\in\Sp_+(d,\bC)$.
    \item There exists a nonzero $T\in\cB(L^2(\R^d))$ such that for every $f\in L^2(\rd)$,
        \begin{equation}\label{eq:full-covariance}
            T\rho(z;\tau)f=\rho(Sz;\tau)Tf, \qquad  z\in\rdd,\; \tau\in\bR,
        \end{equation}
        where the equality holds in $\Shilovp$.
\end{enumerate}
If these conditions hold, the vector space of bounded operators satisfying \eqref{eq:full-covariance} is one-dimensional.
More precisely, if $\widehat S$ is any operator associated with $S\in\Sp_+(d,\bC)$, then \eqref{eq:full-covariance} holds for every $f\in L^2(\rd)$ if and only if $T=c\cdot\widehat S$ for some $c\in\bC$.
\end{theorem}

The proof combines geometric ideas going back to Howe \cite{Howe1988} and H\"ormander \cite{Hormander1995} with an argument that exploits in an essential way the boundedness of the intertwining operator.
The implication $(ii)\Rightarrow(i)$ is the core of the result.
Starting from \eqref{intro.finite-intertwining}, we first differentiate the intertwining relation \eqref{intro.finite-intertwining} and apply the resulting identities to Siegel Gaussians, i.e., functions of the form $\Phi_M(t)=e^{i\pi Mt\cdot t}$ with $M\in\bC^{d\times d}$ symmetric and $\Im(M)>0$.
Their annihilation equations determine positive Lagrangian subspaces, and allow us to prove that $T$ maps every Siegel Gaussian into another Siegel Gaussian, a fact that is well-known for the metaplectic representation \cite{Folland}.
In particular, the existence of a bounded intertwiner implies that the M\"obius action induced by $S$, that is $\cM_S(Z)=(C+DZ)(A+BZ)^{-1}$, preserves the Siegel upper half-space $\HH_d$.

This property alone does not yet yield the full positivity condition defining $\Sp_+(d,\bC)$.
The crucial additional information comes from the boundedness of $T$: after extending the intertwining relation to complex phase-space parameters on Siegel Gaussians and comparing $L^2$-norms, we obtain the positivity inequality for $S$, up to an error depending on the Gaussian.
Varying the Gaussian makes this error arbitrarily small, and hence yields $S\in\Sp_+(d,\bC)$.

The converse implication also strengthens Proposition \ref{intro.propGRT}, where the intertwining relation is proved on the dense subspace generated by real time-frequency shifts of positive Gaussians.
Here, the continuity of the complex Schr\"odinger action on $\Shilovp$ allows us to pass to the $L^2$-closure and hence to obtain \eqref{eq:full-covariance} for every $f\in L^2(\rd)$.
Notice in particular that, although $\rho(Sz;\tau)$ is not in general a bounded operator on $L^2$, its action on the range of $\widehat S$ is represented by the $L^2$-function $\widehat S\rho(z;\tau)f$.

Finally, the one-dimensionality statement follows from the irreducibility of the Schr\"odinger representation.

Theorem \ref{thm:main} therefore contains an intrinsic characterization for positive complex symplectic matrices:
\begin{equation}
    {\Sp}_+(d,\bC)= \left\{S\in\Sp(d,\bC):\text{$S$ admits a nonzero bounded Schr\"odinger intertwiner}\right\}.
\end{equation}
Moreover, for every such $S$ the corresponding intertwining space is one-dimensional.
Thus the Schr\"odinger representation determines directly the operator associated with $S$, up to a constant, and the usual metaplectic semigroup is recovered by choosing the standard two-fold normalization of these projective lines.
{We emphasize the difference with the descriptions of $\Mp_+(d,\bC)$ recalled above, such as \cite{Howe1988,Hormander1995,DerezinskiKarczmarczyk}: there the semigroup is parametrized from the inside, by prescribing the Gaussian kernel or the Weyl symbol of its elements, whereas here it is singled out from the outside, by a property that an arbitrary bounded operator on $L^2(\rd)$ either does or does not satisfy.}
In this sense, Theorem \ref{thm:main} provides for the oscillator semigroup the analogue of the classical Schr\"odinger-representation characterization of the real metaplectic group.

\subsection*{Outline}
The paper is organized as follows.
In Section~\ref{sec:preliminaries} we collect the preliminary material and establish the notation.
In Section~\ref{sec:complex-rho}, we define the complexified Schr\"odinger action on $\Shilovp$, its holomorphic dependence, and identify its infinitesimal form.
Section~\ref{sec:necessity-main} is devoted to the implication from Schr\"odinger intertwining to positivity.
In Section~\ref{sec:characterization}, we prove the converse implication by extending the known dense intertwining relation for the oscillator semigroup to all of $L^2(\R^d)$.
We then prove uniqueness of bounded intertwiners up to a nonzero scalar and obtain the characterization stated in Theorem~\ref{thm:main}.
Finally, Appendix~\ref{appendix:A} contains the proof of the exponential-multiplier lemma for $\Shilov$ used in the construction of the complex Schr\"odinger action.

\section{Prerequisites}\label{sec:preliminaries}
We denote by $x\cdot\xi$ the standard inner product in $\rd$.
If $\alpha,\beta\in\bN^d$ are multi-indices, the notation $\alpha\leq\beta$ means that $\alpha_j\leq\beta_j$ for every $j=1,\ldots,d$.
If $v\in\bC^d$, we shall denote by $\Re(v)$ and $\Im(v)$ its real and imaginary parts, respectively.
A similar notation will be used for matrices $M\in\bC^{d\times d}$.
We denote by $M^\top$ the transpose of $M$, and the set of complex symmetric matrices is $\Sym(d,\bC)$.
The group of $d\times d$ complex invertible matrices is $\GL(d,\bC)$.
Similar notations will be used for real matrices.
The Siegel upper half-space, denoted by $\HH_d$, is the set of matrices $M\in\Sym(d,\bC)$ so that $\Im(M)>0$.
A {\em Siegel Gaussian} is a function $\Phi_M(t)=e^{i\pi Mt\cdot t}$ with $M\in\HH_d$.
We denote by $I$ and $O$ the identity and the zero-matrix, respectively.
We denote by $\mathcal{B}(L^2(\rd))$ the space of bounded linear operators on $L^2(\rd)$, and by $\mathcal{U}(L^2(\rd))$ the space of unitary operators on $L^2(\rd)$.
If $T\in\mathcal{B}(L^2(\rd))$, its operator norm will be denoted by $\norm{T}_{\op}$.
We also denote by $\mathcal{D}'(\rd)$ the topological dual of the space of compactly-supported smooth functions on $\rd$.

We now delve into the symplectic and metaplectic groups, and the other tools we need for our analysis.
For the theory of symplectic matrices and metaplectic operators, we refer to \cite{CorderoBook} or \cite{deGosson}.
The representation theory is summarized in \cite{Berge2022} and \cite{Folland}.
\subsection{Symplectic matrices}
A matrix $S\in\bC^{2d\times2d}$ is symplectic if $S^\top JS=J$, where $J$ is defined as in \eqref{intro.defJ}.
We write $S\in\Sp(d,\bC)$.
By partitioning $S$ into $d\times d$ complex blocks
\begin{equation}\label{blockS}
    S=\begin{pmatrix} A & B\\
        C & D \end{pmatrix}, \qquad A,B,C,D\in\bC^{d\times d},
\end{equation}
and using \eqref{intro.defJ}, it is easy to see that $S\in\Sp(d,\bC)$ if and only if its blocks satisfy
\begin{equation}
    \label{eq:symplectic-block-identities} A^\top C=C^\top A,\qquad B^\top D=D^\top B, \qquad  A^\top D-C^\top B=I.
\end{equation}
Every symplectic matrix $S\in\Sp(d,\bC)$ is invertible and $\det(S)=1$.
We denote by $\Sp(d,\bR)=\Sp(d,\bC)\cap\bR^{2d\times2d}$, and by $\Sp_+(d,\bC)$ the semigroup of complex positive symplectic matrices, see Definition \ref{intro.defPos}.
Let us denote by $\bF$ either $\bR$ or $\bC$.
Then, the group $\Sp(d,\bF)$ is generated by $J$ and by the subgroups of matrices of the form
\begin{equation}\label{defDE}
    \cD_E=\begin{pmatrix} E^{-1} & O\\
        O & E^\top \end{pmatrix}, \qquad E\in\GL(d,\bF)
\end{equation}
and
\begin{equation}\label{defVQ}
    V_Q=\begin{pmatrix} I & O\\
        Q & I \end{pmatrix}, \qquad Q\in\Sym(d,\bF).
\end{equation}

\subsection{Tools from representation theory}
A {\em unitary representation} of a locally compact group $G$ on $L^2(\rd)$ is a group homomorphism $\pi:G\to\mathcal{U}(L^2(\rd))$ so that the mapping $z\in G\mapsto \pi(z)f\in L^2(\rd)$ is continuous for every $f\in L^2(\rd)$.
The {\em cyclic subspace} generated by $f\in L^2(\rd)$ is
\begin{equation}
    \mathcal O_f:=\overline{\mathrm{span}\{\pi(g)f:g\in G\}}\subseteq L^2(\rd).
\end{equation}
If $\mathcal O_f=L^2(\rd)$, we say that $f$ is {\em cyclic}.
The representation $\pi$ is {\em irreducible} when every non-zero function $f\in L^2(\rd)$ is cyclic.

The representation used in this work is the Schr\"odinger representation of the Heisenberg group, defined in \eqref{intro.Schrodingerrep}, which is a unitary and irreducible representation \cite{Folland}.

For later use, the infinitesimal differential expression of \eqref{intro.Schrodingerrep} is
\begin{equation}\label{eq:drho-formal}
  \ud\rho(y,\eta)=-y\cdot\nabla+2\pi i\eta\cdot t, \qquad (y,\eta)\in\bC^{2d},
\end{equation}
initially on $\cS(\R^d)$ and extended complex-linearly in $(y,\eta)$.

\subsection{Metaplectic operators}
For a matrix $U\in\Sp(d,\bR)$,
\begin{equation}
    \rho_U(z;\tau)=\rho(Uz;\tau), \qquad z\in\rdd,\;\tau\in\bR
\end{equation}
defines another irreducible unitary representation of the Heisenberg group, satisfying $\rho_U(0;\tau)=\rho(0;\tau)$ for every $\tau\in\bR$.
By Stone-von Neumann theorem \cite{Folland}, there exists an operator $\widehat U\in\mathcal{U}(L^2(\rd))$ so that the intertwining relation
\begin{equation}\label{intertU}
    \widehat U\rho(z;\tau)\widehat U^{-1}=\rho(Uz;\tau), \qquad z\in\rdd,\;\tau\in\bR
\end{equation}
holds.
The operators satisfying \eqref{intertU} are not unique, however if $\widehat U'$ is another unitary operator satisfying it, then $\widehat U'=c\cdot\widehat U$ for some $c\in\bC$ with $|c|=1$.
These are called {\em metaplectic operators} and form a group under composition.
This group has a subgroup that contains precisely two operators for each $U\in\Sp(d,\bR)$, differing by a sign, called {\em metaplectic group} and herein denoted by $\Mp(d,\bR)$.
The {\em projection} $\pi^{Mp}:\widehat U\in\Mp(d,\bR)\mapsto U\in\Sp(d,\bR)$ is a group homomorphism with kernel $\{\pm\mathrm{id}_{L^2}\}$, and it is a realization of the two-fold cover of $\Sp(d,\bR)$.
The relation \eqref{intertU} characterizes the fibers of the generators of $\Sp(d,\bR)$, hence the generators of $\Mp(d,\bR)$.
Indeed, if $J$ is as in \eqref{intro.defJ}, $\cD_E$ and $V_Q$ are as in \eqref{defDE} and \eqref{defVQ}, respectively, then
\begin{align}\label{defFT}
    \Fur f(\xi)=\widehat f(\xi)=i^{-d/2}\int_{\rd}f(t)e^{-2\pi i\xi\cdot t}\ud t, \qquad f\in\cS(\rd),
\end{align}
the Fourier transform, has $\pi^{Mp}(\Fur)=J$, whereas the operators
\begin{align}\label{defTE}
    &\mathfrak{T}_Ef(t)=i^m|\det(E)|^{1/2}f(Et), \qquad E\in \GL(d,\bR)\\
    \label{defpQ} &\mathfrak{p}_Qf(t)=\Phi_Q(t)f(t), \qquad Q\in\Sym(d,\bR),
\end{align}
defined for $f\in L^2(\rd)$, where $\Phi_Q(t)=e^{i\pi Qt\cdot t}$, satisfy $\pi^{Mp}(\mathfrak T_E)=\cD_E$ and $\pi^{Mp}(\mathfrak p_Q)=V_Q$.
In \eqref{defTE}, $m\pi\equiv\arg\det(E)$ (mod $2\pi$).

\subsection{The oscillator semigroup}
The oscillator semigroup is the two-fold cover of $\Sp_+(d,\bC)$.
The two-fold cover, which is also known as {\em projection}, $\pi^{Mp}_+:\widehat S\in\Mp_+(d,\bC)\to S\in\Sp_+(d,\bC)$ is a semigroup homomorphism which coincides with $\pi^{Mp}$ on $\Mp(d,\bR)$.
It is proved in \cite{GRT} that the generators of $\Mp_+(d,\bC)$ are precisely the Fourier transform \eqref{defFT} and the operators \eqref{defTE} with $E\in\GL(d,\bR)$, together with the operators in \eqref{defpQ}, with $\Sym(d,\bR)$ being replaced by $\HH_d$.
Moreover, if $Q\in \HH_d$, and
\begin{equation}
    \mathfrak p_Qf(t)=\Phi_Q(t)f(t), \qquad f\in L^2(\rd),
\end{equation}
then $\pi_+^{Mp}(\mathfrak p_Q)=V_Q$, defined as in \eqref{defVQ}.
This characterizes completely a realization of $\pi_+^{Mp}$, because it defines it on the generators on $\Mp_+(d,\bC)$.

\subsection{Positive Lagrangians, the Siegel space, and Gaussian states}\label{sec:lagrangian-prelim}
For the theory in this section, we refer to the work by H\"ormander \cite{Hormander1995}.
{See also \cite{Hormander1983,MelinSjostrand} for the role of positive Lagrangian subspaces in the calculus of Fourier integral operators with complex phase.}

A complex subspace $\cL\subset\bC^{2d}$ is Lagrangian if $\dim_\bC \cL=d$ and $\sigma|_{\cL\times \cL}=0$.
It is \emph{positive} if
\begin{equation}\label{eq:positive-lagrangian}
  i^{-1}\sigma(\zeta,\bar\zeta)>0, \qquad 0\neq\zeta\in \cL.
\end{equation}
We denote the set of positive Lagrangian subspaces by $\Lambda_d^+$.
For $Z\in\HH_d$ define
\begin{equation}\label{eq:LZ}
  \cL_Z:=\{(y,Zy):y\in\bC^d\}\subset\bC^{2d}.
\end{equation}
The Gaussian $\Phi_Z(t)=e^{i\pi Zt\cdot t}$ belongs to $\cS(\R^d)$, and
\begin{equation}\label{eq:gaussian-annihilation-prelim}
  \ud\rho(y,Zy)\Phi_Z=0, \qquad y\in\bC^d.
\end{equation}
The special real Gaussian $\f_M(t)=e^{-\pi Mt\cdot t}$ corresponds to $Z=iM$ with $M=M^\top>0$.

\begin{proposition}\label{prop:lagrangian-gaussian}
The map $Z\mapsto \cL_Z$ is a bijection from $\HH_d$ onto $\Lambda_d^+$.
Moreover, if $\cL=\cL_Z\in\Lambda_d^+$, then
\begin{equation}\label{eq:annihilator-line}
  \{u\in\cS'(\R^d):\ud\rho(\zeta)u=0\ \text{for every }\zeta\in \cL\} =\bC\Phi_Z.
\end{equation}
In particular, a positive Lagrangian determines a unique Gaussian line.
\end{proposition}
This is the Gaussian/positive-Lagrangian characterization used throughout the paper; compare H\"ormander \cite[Proposition~5.1]{Hormander1995} and Howe \cite[Theorem~22.2]{Howe1988}.
Note that Howe uses the terminology {\em positive polarizations} instead of positive Lagrangians.

\subsection{The Gelfand-Shilov space and holomorphic maps}\label{sec:GS}
We use the critical Roumieu Gelfand-Shilov space $\Shilov$.
{We refer to \cite{GelfandShilov} for the general theory of these spaces and to \cite{NicolaRodino} for their role in global pseudodifferential calculus.}
Following \cite[Definition~2.1(v)]{Petersson}, a function $\varphi\in C^\infty(\R^d)$ belongs to $\Shilov$ if there exist $C,h>0$ such that
\begin{equation}\label{eq:GS-estimate}
  \sup_{t\in\R^d} |t^\beta\partial^\gamma\varphi(t)| \leq C h^{|\beta+\gamma|}(\beta!\gamma!)^{1/2}, \qquad \beta,\gamma\in\N^d.
\end{equation}
For $h>0$, let $\Shilovh$ be the Banach space with norm
\begin{equation}\label{eq:Gh-norm}
  \|\varphi\|_h:=\sup_{t\in\R^d,\, \beta,\gamma\in\N^d} \frac{|t^\beta\partial^\gamma\varphi(t)|}{h^{|\beta+\gamma|}(\beta!\gamma!)^{1/2}}.
\end{equation}
Then
\begin{equation}
  \Shilov=\bigcup_{h>0}\Shilovh
\end{equation}
with its Roumieu inductive-limit topology.
Its topological dual is the space of ultra-tempered distributions.
We denote by $\langle u,\varphi\rangle$ the bilinear distributional pairing $\Shilovp\times\Shilov$.

The choice of the Roumieu space at the critical indices is essential: \cite[Proposition~6.3 and Theorem~6.4]{Petersson} imply $\Shilov\neq\{0\}$ whereas the corresponding critical Beurling space is trivial.
By \cite[Proposition~2.3(a)]{Petersson}, the Fourier transform preserves $\Shilov$.

Since $\Shilov\subset\cS(\rd)\subset L^2(\rd)$ continuously, every $f\in L^2(\R^d)$ defines an element of $\Shilovp$ by duality
\begin{equation}\label{eq:L2-in-Gprime}
  \langle f,\varphi\rangle=\int_{\R^d}f(t)\varphi(t)\ud t.
\end{equation}
Finite linear combinations of Hermite functions belong to $\Shilov$ and are dense in $\cS(\R^d)$ in the Schwartz topology; see e.g., \cite{ToftKhrennikovNilssonNordebo}.
Consequently, we have the following density result.

\begin{lemma}\label{lem:G-dense}
The inclusion $\Shilov\hookrightarrow\cS(\R^d)$ has dense range.
Hence the restriction map
\begin{equation}
  \cS'(\R^d)\longrightarrow\Shilovp
\end{equation}
is injective.
\end{lemma}

\subsection{Holomorphic maps with values in the Gelfand-Shilov space}

Because $\Shilov$ is not a Banach space, we make explicit the meaning of holomorphic parameter dependence.

\begin{definition}\label{def:G-entire}
Let $E$ be a complex locally convex space.
A map $F:\bC\to E$ is \emph{holomorphic} if, for every $\zeta_0\in\bC$, the limit
\begin{equation}\label{eq:lcs-derivative}
  F'(\zeta_0) =\lim_{\delta\to0} \frac{F(\zeta_0+\delta)-F(\zeta_0)}{\delta}
\end{equation}
exists in the topology of the completion of $E$, where the increment $\delta$ belongs to $\bC$.
It is \emph{entire} if it is holomorphic on all of $\bC$.
\end{definition}
In our applications, $E=\Shilov$, endowed with its Roumieu inductive-limit topology.
Since $\Shilov$ is complete, its completion $\widehat E$ in the sense of \cite[Definition~3.1]{Kruse} can be canonically identified with $E$ itself.
Hence, in our setting, the above limit may simply be understood as convergence in $\Shilov$.

In the concrete maps below, we use the stronger fact that locally in the complex parameter $\zeta$ the Taylor series converges uniformly in one fixed Banach step $\Shilovh$.
For the first part of the statement, we also refer to \cite[Theorem~3.11]{CappielloToftGS}.
The proof is detailed in Appendix \ref{appendix:A}.

\begin{lemma}\label{lem:exp-multiplier}
Let $\lambda\in\bC^d$ and $\varphi\in\Shilov$.
Multiplication by $e^{\lambda\cdot t}$ is a continuous automorphism of $\Shilov$.
Moreover,
\begin{equation}\label{eq:exp-entire}
  \bC\ni\zeta\longmapsto e^{\zeta\lambda\cdot t}\varphi(t)\in\Shilov
\end{equation}
is entire.
More precisely, for every $\zeta_0\in\bC$ and $R>0$ there exists $h>0$ such that
\begin{equation}\label{eq:exp-series}
  e^{\zeta\lambda\cdot t}\varphi(t) =\sum_{n=0}^\infty \frac{(\zeta-\zeta_0)^n}{n!} (\lambda\cdot t)^n e^{\zeta_0\lambda\cdot t}\varphi(t)
\end{equation}
converges uniformly for $|\zeta-\zeta_0|\leq R$ in $\Shilovh$.
In particular,
\begin{equation}\label{eq:exp-derivative}
  \frac{\ud}{\ud\zeta} \big(e^{\zeta\lambda\cdot t}\varphi(t)\big) =(\lambda\cdot t)e^{\zeta\lambda\cdot t}\varphi(t) \quad\text{in }\Shilov.
\end{equation}
\end{lemma}

\section{The complex Schr\"odinger action on \texorpdfstring{$\Shilovp$}{Shilov}}\label{sec:complex-rho}

For $x,\xi\in\bC^d$, define on $\Shilov$ the complex modulation and translation
\begin{align}
   &M_\xi\varphi(t)=e^{2\pi i\xi\cdot t}\varphi(t),\label{eq:M-test}\\
 & T_x\varphi=\Fur^{-1}\big(e^{-2\pi i x\cdot\omega}\widehat\varphi(\omega)\big). \label{eq:T-test}
\end{align}
By Lemma~\ref{lem:exp-multiplier} and Fourier invariance, both are continuous automorphisms of $\Shilov$.
We extend them to $\Shilovp$ by duality
\begin{align}
  \langle M_\xi u,\varphi\rangle &:=\langle u,M_\xi\varphi\rangle,\label{eq:M-dual}\\
  \langle T_xu,\varphi\rangle &:=\langle u, T_{-x}\varphi\rangle. \label{eq:T-dual}
\end{align}
For real parameters these are the usual modulation and translation of distributions.
For $x,\xi\in\bC^d$ and $\tau\in\bR$, we then set
\begin{equation}\label{eq:rho-Gprime}
  \rho(x,\xi;\tau) :=e^{2\pi i\tau}e^{-i\pi x\cdot\xi} M_\xi T_x \quad\text{on }\Shilovp.
\end{equation}
Thus the same symbol $\rho$ denotes the real Schr\"odinger representation and its complex extension, as it shall not cause confusion.

Since functions in $\Shilov$ display Gaussian decay, together with their Fourier transforms, the following representation result is evident.

\begin{proposition}
\label{prop:agreement}
For $f\in\Shilov$, $x,\xi\in\bC^d$, and $\tau\in\R$, the distribution \eqref{eq:rho-Gprime} is represented by
\begin{equation}
  \rho(x,\xi;\tau)f(t) = e^{2\pi i\tau}e^{-i\pi x\cdot\xi}e^{2\pi i\xi\cdot t} \int_{\R^d}\widehat f(\omega) e^{2\pi i\omega\cdot(t-x)}\ud\omega.
\end{equation}
\end{proposition}

We now discuss the infinitesimal version of our intertwining relation \eqref{intro.finite-intertwining}.
\begin{proposition}\label{prop:weak-diff}
Let $u\in\Shilovp$ and let $(y,\eta)\in\bC^{2d}$ be fixed.
The map
\begin{equation}\label{eq:weak-entire-orbit}
  \bC\ni\zeta\longmapsto \rho(\zeta y,\zeta\eta;0)u\in\Shilovp
\end{equation}
is weakly entire, meaning that its pairing with every fixed $\varphi\in\Shilov$ is an entire scalar function of $\zeta\in\bC$.
In particular, its restriction to the real axis
\begin{equation}
  \bR\ni r\longmapsto\rho(ry,r\eta;0)u
\end{equation}
is weakly $C^1$ and
\begin{equation}\label{eq:drho-Gprime}
  \left.\frac{\ud}{\ud r}\right|_{r=0} \rho(ry,r\eta;0)u =-y\cdot\nabla u+2\pi i\eta\cdot tu \quad\text{in }\Shilovp.
\end{equation}
\end{proposition}

\begin{proof}
For fixed $\varphi\in\Shilov$, the maps $\zeta\mapsto M_{\zeta\eta}\varphi$ and $\zeta\mapsto T_{\zeta y}\varphi$ are entire $\Shilov$-valued maps by Lemma~\ref{lem:exp-multiplier} and Fourier invariance.
The scalar factor $e^{-i\pi\zeta^2y\cdot \eta}$ is entire, hence the pairing of \eqref{eq:weak-entire-orbit} with $\varphi$ is also entire.
For the derivative, restrict to $r\in\R$.
Translation gives
\begin{align*}
 \left\langle \left.\frac{\ud}{\ud r}\right|_{r=0}T_{ry}u,\varphi \right\rangle =\left\langle u, \left.\frac{\ud}{\ud r}\right|_{r=0}T_{-ry}\varphi \right\rangle =\langle u,y\cdot\nabla\varphi\rangle =\langle-y\cdot\nabla u,\varphi\rangle.
\end{align*}
Similarly,
\begin{equation}
 \left\langle \left.\frac{\ud}{\ud r}\right|_{r=0}M_{r\eta}u,\varphi \right\rangle =\langle u,2\pi i\eta\cdot t\varphi\rangle =\langle2\pi i\,\eta\cdot tu,\varphi\rangle.
\end{equation}
The derivative at zero of the scalar factor $e^{-i\pi r^2y\cdot\eta}$ vanishes, and \eqref{eq:drho-Gprime} follows.
\end{proof}

\section{From Schr\"odinger intertwining to positivity}\label{sec:necessity-main}
The purpose of this section is to prove that the existence of a bounded intertwiner implies positivity.

\begin{theorem}\label{thm:necessity}
Let $S\in\Sp(d,\bC)$.
If there exists a nonzero $T\in\cB(L^2(\R^d))$ satisfying \eqref{eq:full-covariance}, then $S\in\Sp_+(d,\bC)$.
\end{theorem}
The proof of this result requires a deeper study of the action of bounded intertwiners on Gaussians.
As for the oscillator semigroup, these results are known.
However, we stress that we are not assuming operators to be in $\Mp_+(d,\bC)$, for they are only bounded and satisfy the intertwining relation with the Schr\"odinger representation of the complexified Heisenberg group.

\subsection{Gaussians and intertwiners}\label{sec:gaussians}
We first record an elementary consequence of irreducibility \cite[Proposition~1.43]{Folland}.
In our case, this fact can be expressed as follows.

\begin{lemma}\label{lem:T-injective}
Let $S\in\Sp(d,\bC)$ and let $T\in\mathcal{B}(L^2(\mathbb{R}^d))$ be a non-zero operator satisfying \eqref{eq:full-covariance}.
Then $T$ is injective.
\end{lemma}

\begin{proof}
The closed subspace $\ker T$ is invariant under the action $\rho(z;\tau)$, because if $f\in \ker T\setminus\{0\}$, then
\begin{equation}
  T\rho(z;\tau)f=\rho(Sz;\tau)Tf=0, \qquad z\in\mathbb{R}^{2d},\; \tau\in\bR.
\end{equation}
Thus,
\begin{equation}
    \mathcal O_f=\overline{\mathrm{span}\{\rho(z;\tau)f:z\in\mathbb{R}^{2d}\}}\subseteq\ker(T).
\end{equation}
Since the Schr\"odinger representation of the Heisenberg group is irreducible, $f$ is cyclic, and therefore $\ker T=L^2(\mathbb{R}^d)$, which contradicts $T\neq0$.

\end{proof}

We now prove that the existence of bounded intertwiners for the Schr\"odinger representation, with respect to a complex symplectic matrix $S$, implies that the Möbius map associated with $S$ has values in $\HH_d$.
Recall that we denote $\Phi_Z(t)=e^{i\pi Zt\cdot t}$.

\begin{proposition}\label{prop:Gaussian-determination}
Let $S\in\Sp(d,\bC)$ have blocks \eqref{blockS} and let $T\in\cB(L^2(\R^d))$ be a bounded intertwiner of $S$.
Fix $Z\in\HH_d$ and define
\begin{equation}\label{eq:LNZ}
  L_Z:=A+BZ, \qquad N_Z:=C+DZ.
\end{equation}
Then $L_Z$ is invertible, the matrix
\begin{equation}\label{eq:PhiS}
  \mathcal{M}_S(Z):=N_ZL_Z^{-1} =(C+DZ)(A+BZ)^{-1}
\end{equation}
belongs to $\HH_d$, and
\begin{equation}\label{eq:TphiZ}
  T\Phi_Z=c_Z\Phi_{\cM_S(Z)}
\end{equation}
for some $c_Z\in\bC\setminus\{0\}$.
\end{proposition}

\begin{proof}
Put $u=T\Phi_Z$.
By Lemma \ref{lem:T-injective}, $u\neq0$.
Write $Z=X+iY$, with $X=X^\top$ and $Y=Y^\top>0$.
For a real direction $v\in\R^{2d}$, differentiate
\begin{equation}
  T\rho(rv;0)\Phi_Z=\rho(rSv;0)u, \qquad r\in\R,
\end{equation}
at $r=0$.
The left-hand side is differentiable in $L^2$; the right-hand side is differentiated in $\Shilovp$ by Proposition \ref{prop:weak-diff}.
We obtain
\begin{equation}\label{eq:diff-covariance-Z}
  T\ud\rho(v)\Phi_Z=\ud\rho(Sv)u, \qquad v\in\R^{2d},
\end{equation}
where as usual
\begin{equation}\label{eq:drhoeta}
  \ud\rho(y,\eta)=-y\cdot\nabla+2\pi i\eta\cdot t
\end{equation}
is used complex-linearly for $(y,\eta)\in\bC^{2d}$.
Fix $y\in\R^d$, and set $v_1=(y,Xy)$ and $v_2=(0,Yy)$.
Since
\begin{equation}
  \nabla\Phi_Z(t)=2\pi iZt\Phi_Z(t),
\end{equation}
and $v=v_1+iv_2=(y,Zy)$, we have
\begin{equation}\label{eq:Z-annihilation}
  \ud\rho(v_1)\Phi_Z+i\ud\rho(v_2)\Phi_Z =\ud\rho(y,Zy)\Phi_Z=0.
\end{equation}
Apply \eqref{eq:diff-covariance-Z} to $v_1$ and $v_2$, multiply the second identity by $i$, and add.
Then, by \eqref{eq:Z-annihilation}, we obtain
\begin{equation}\label{eq:transformed-annihilation-Z}
  \ud\rho\big(S(y,Zy)\big)u=0, \qquad y\in\R^d.
\end{equation}
Since
\begin{equation}
  S\binom{y}{Zy}=\binom{L_Zy}{N_Zy},
\end{equation}
we get, by \eqref{eq:drhoeta}, that
\begin{equation}\label{eq:system-Z-withy}
    -y\cdot\big(-L_Z^\top \nabla u+2\pi iN_Z^\top tu\big)=0, \qquad y\in\mathbb{R}^d,
\end{equation}
that is
\begin{equation}\label{eq:system-Z}
  L_Z^\top\nabla u=2\pi i N_Z^\top t u.
\end{equation}
The equality is first obtained in $\Shilovp$.
Since $u\in L^2$, both sides are tempered distributions, and therefore Lemma \ref{lem:G-dense} yields the equality in $\cS'(\rd)$.

We now prove that $L_Z$ is invertible.
If $L_Zy=0$ for some $y\in\bC^d\setminus\{0\}$, then $N_Zy\neq0$ because $S$ is invertible and $(y,Zy)\neq0$.
By \eqref{eq:system-Z-withy},
\begin{equation}
  (N_Zy)\cdot t u=0 \quad\text{in }\cS'(\rd).
\end{equation}
As $u\in L^2$, this implies that $u$ vanishes almost everywhere outside the zero set of the nonzero complex linear form $t\mapsto(N_Zy)\cdot t$.
That zero set is a proper real linear subspace and has Lebesgue measure zero.
Hence $u=0$, a contradiction.

Set $Q=N_ZL_Z^{-1}$.
The symplectic block identities \eqref{eq:symplectic-block-identities} and the symmetry $Z^\top=Z$ give
\begin{equation}
  L_Z^\top N_Z=N_Z^\top L_Z,
\end{equation}
whence
\begin{equation}\label{eq:Q-sym}
  Q^\top=Q.
\end{equation}
Multiplying \eqref{eq:system-Z} by $L_Z^{-\top}$ yields
\begin{equation}\label{eq:gradient-Q}
  \nabla u=2\pi i Qt u \quad\text{in }\cS'(\rd).
\end{equation}
We regard $u\in\cS'(\R^d)$ as an element of $\mathcal{D}'(\R^d)$ and set
\begin{equation}
  v=e^{-i\pi Qt\cdot t}u\in\mathcal{D}'(\R^d).
\end{equation}
Since $Q=Q^\top$,
\begin{equation}
  \partial_{t_j}(t\cdot Qt)=2(Qt)_j, \qquad j=1,\ldots,d,
\end{equation}
and Leibniz rule gives
\begin{equation}
  \partial_{t_j}v=-2\pi i(Qt)_j e^{-i\pi Qt\cdot t}u +e^{-i\pi Qt\cdot t}\partial_{t_j}u, \qquad j=1,\ldots,d.
\end{equation}
Using
\begin{equation}
  \partial_{t_j}u=2\pi i(Qt)_j u, \qquad j=1,\ldots,d,
\end{equation}
we obtain
\begin{equation}
  \partial_{t_j}v=0,\qquad j=1,\ldots,d.
\end{equation}
Hence $\nabla v=0$ in $\mathcal{D}'(\R^d)$.
By the standard distribution-theoretic fact that a distribution on a connected open set whose first-order derivatives all vanish is constant, and since $\bR^d$ is connected, there exists $c\in\bC$ such that
\begin{equation}
  v=c \qquad\text{in }\mathcal{D}'(\R^d).
\end{equation}
Multiplying by the smooth inverse $e^{i\pi Qt\cdot t}$ yields
\begin{equation}
  u=c\cdot e^{i\pi Qt\cdot t}=c\cdot \Phi_Q \qquad\text{in }\mathcal{D}'(\R^d).
\end{equation}
Hence,
\begin{equation}
  u(t)=c_Ze^{i\pi Qt\cdot t}
\end{equation}
with $c_Z\neq0$.
Since $u$ is a nonzero $L^2$ function and $Q^\top=Q$, it follows that $\Im(Q)>0$.
Therefore $Q\in\HH_d$, proving \eqref{eq:PhiS}--\eqref{eq:TphiZ}.

\end{proof}

Therefore, the intertwining relation, together with the boundedness of $T$, implies that
\begin{equation}\label{eq:Phi-self-map}
  \mathcal{M}_S(\HH_d)\subseteq\HH_d.
\end{equation}

We now continue holomorphically the image of Gaussians through the complexified Schr\"odinger representation, and establish the intertwining relation on Gaussian states.

\begin{lemma}\label{lem:complex-shift-gaussian}
Let $Z\in\HH_d$.
For every $x,\xi\in\bC^d$,
\begin{equation}\label{eq:rho-phiZ-explicit}
  \rho(x,\xi;0)\Phi_Z(t) =\exp\!\Big( i\pi Zt\cdot t +2\pi i(\xi-Zx)\cdot t +i\pi(x\cdot Zx-x\cdot\xi) \Big).
\end{equation}
In particular, $\rho(x,\xi;0)\Phi_Z\in L^2(\R^d)$ for all complex $x,\xi$, and
\begin{equation}
  \bC^{2d}\ni(x,\xi)\longmapsto \rho(x,\xi;0)\Phi_Z\in L^2(\R^d)
\end{equation}
is entire as an $L^2$-valued map.
\end{lemma}

\begin{proof}
Formula \eqref{eq:rho-phiZ-explicit} follows directly from \eqref{eq:rho-Gprime} and Gaussian integration.
If $Z=X+iY$ with $Y>0$, the quadratic part of the real exponent in the modulus is $-\pi Yt\cdot t$.
Complex phase-space parameters only contribute linear terms in $t$ and a scalar factor.
Hence the function is in $L^2$.

Let $Z=X+iY\in\HH_d$, and write $\zeta=(x,\xi)\in\bC^{2d}$.
By \eqref{eq:rho-phiZ-explicit}, the map $\zeta\mapsto\rho(\zeta;0)\Phi_Z(t)$ is entire for every fixed $t\in\R^d$.
We show that it is entire with values in $L^2(\R^d)$.

Let $K\subset\bC^{2d}$ be compact.
From \eqref{eq:rho-phiZ-explicit} there exist $c>0$ and $C_K>0$ such that
\begin{equation}
  |\rho(\zeta;0)\Phi_Z(t)|\leq C_K e^{-c|t|^2+C_K|t|}, \qquad \zeta\in K.
\end{equation}
Moreover, each first complex partial derivative with respect to the phase-space variables is a linear polynomial in $t$, with coefficients bounded on $K$, times $\rho(\zeta;0)\Phi_Z(t)$.
Hence
\begin{equation}
  \bigl|\partial_{\zeta_j} \bigl(\rho(\zeta;0)\Phi_Z(t)\bigr)\bigr| \leq C_K(1+|t|)e^{-c|t|^2+C_K|t|}, \qquad \zeta\in K,
\end{equation}
and the right-hand side belongs to $L^2(\R^d)$.
For fixed $\zeta_0\in \bC^{2d}$, coordinate direction $e_j$ and $h>0$ small enough that $\{\zeta_0+\theta he_j:0\leq\theta\leq1\}$ is contained in $K$,
\begin{equation}
  \frac{ \rho(\zeta_0+he_j;0)\Phi_Z(t) -\rho(\zeta_0;0)\Phi_Z(t) }{h} = \int_0^1 \partial_{\zeta_j} \bigl(\rho(\zeta_0+\theta he_j;0)\Phi_Z(t)\bigr) \ud\theta
\end{equation}
pointwise in $t$.
Dominated convergence therefore shows that the first complex partial derivatives exist and depend continuously on $\zeta$ as $L^2$-valued maps.
The resulting differential is complex linear, so
\begin{equation}
  \zeta\longmapsto\rho(\zeta;0)\Phi_Z
\end{equation}
is holomorphic with values in $L^2(\R^d)$.
Since $\zeta$ is arbitrary, the map is entire.

\end{proof}

\begin{proposition}\label{prop:complex-covariance-gaussian}
Under the assumptions of Proposition \ref{prop:Gaussian-determination}, for every fixed $Z\in\HH_d$,
    \begin{equation}\label{eq:complex-covariance-gaussian}
         T\rho(\zeta;0)\Phi_Z=\rho(S\zeta;0)T\Phi_Z
    \end{equation}
holds in $L^2(\R^d)$ for every $\zeta\in\bC^{2d}$.
\end{proposition}
\begin{proof}
By Proposition \ref{prop:Gaussian-determination}, both $\Phi_Z$ and $T\Phi_Z=c_Z\Phi_{\mathcal{M}_S(Z)}$ are Siegel Gaussians.
Thus both sides of \eqref{eq:complex-covariance-gaussian} are entire $L^2$-valued functions of $\zeta\in\bC^{2d}$ by Lemma \ref{lem:complex-shift-gaussian}.
They agree on $\R^{2d}$ by the assumed intertwining relation.
Composing their difference with an arbitrary continuous complex-linear functional on $L^2$ gives a scalar entire function on $\bC^{2d}$ vanishing on $\R^{2d}$.
The one-variable identity theorem, applied successively to the coordinates, gives the result.
\end{proof}

Let $Z=X+iY\in\HH_d$.
For a real phase-space vector $z_1=(x_1,\xi_1)\in\R^{2d}$, set
\begin{equation}\label{eq:gZ}
  g_Z(z_1) :=x_1\cdot Yx_1 +(\xi_1-Xx_1)\cdot Y^{-1}(\xi_1-Xx_1).
\end{equation}
If $\zeta\in\bC^{2d}$, write uniquely
\begin{equation}\label{eq:zeta-real-imag}
  \zeta=z_0+iz_1, \qquad z_0,z_1\in\R^{2d},
\end{equation}
and define
\begin{equation}\label{eq:FZ}
  F_Z(\zeta):=g_Z(z_1)+\sigma(z_0,z_1).
\end{equation}
The penultimate ingredient for the proof of Theorem \ref{thm:necessity} is the following exact norm growth of the complexified Schr\"odinger representation on Gaussians.

\begin{lemma}\label{lem:norm-formula}
Let $Z\in\HH_d$ and $\zeta\in\bC^{2d}$.
Then
\begin{equation}\label{eq:norm-formula}
  \|\rho(\zeta;0)\Phi_Z\|_2^2 =\|\Phi_Z\|_2^2 e^{2\pi F_Z(\zeta)}.
\end{equation}
Consequently, for every $r\in\R$,
\begin{equation}\label{eq:norm-growth-r}
  \|\rho(r\zeta;0)\Phi_Z\|_2^2 =\|\Phi_Z\|_2^2 e^{2\pi r^2F_Z(\zeta)}.
\end{equation}
\end{lemma}

\begin{proof}
Write $\zeta=(x,\xi)$ with
\begin{equation}
  x=x_0+ix_1,\qquad \xi=\xi_0+i\xi_1, \qquad x_j,\xi_j\in\R^d,
\end{equation}
so that $z_j=(x_j,\xi_j)$.
By \eqref{eq:rho-phiZ-explicit}, let
\begin{equation}
  c=\Im(\xi-Zx), \qquad d_0=\Im(x\cdot Zx-x\cdot\xi).
\end{equation}
Then
\begin{equation}
 |\rho(x,\xi;0)\Phi_Z(t)|^2 =e^{-2\pi Yt\cdot t-4\pi c\cdot t-2\pi d_0}.
\end{equation}
Completing the square yields
\begin{equation}\label{eq:norm-intermediate}
  \frac{\|\rho(x,\xi;0)\Phi_Z\|_2^2} {\|\Phi_Z\|_2^2} =\exp\!\Big(2\pi(c\cdot Y^{-1}c-d_0)\Big).
\end{equation}
A direct expansion using $Z=X+iY$ gives
\begin{align*}
 c\cdot Y^{-1}c-d_0 &=x_1\cdot Yx_1 +(\xi_1-Xx_1)\cdot Y^{-1}(\xi_1-Xx_1)+\xi_0\cdot x_1-x_0\cdot\xi_1\\
 &=g_Z(z_1)+\sigma(z_0,z_1)=F_Z(\zeta).
\end{align*}
This proves \eqref{eq:norm-formula}.
Since $F_Z$ is homogeneous of degree two under real scaling, \eqref{eq:norm-growth-r} follows.
\end{proof}

We are now ready to prove this squeezing result for the Gaussian metric $g_Z$.

\begin{lemma}\label{lem:squeeze}
For every fixed $z_1\in\R^{2d}$,
\begin{equation}\label{eq:inf-gZ}
  \inf_{Z\in\HH_d}g_Z(z_1)=0.
\end{equation}
\end{lemma}

\begin{proof}
Write $z_1=(x_1,\xi_1)$.
If $x_1\neq0$, choose a real symmetric matrix $X$ such that $Xx_1=\xi_1$, such as,
\begin{equation}\label{eq:X-explicit}
  X =\frac{\xi_1x_1^\top+x_1\xi_1^\top}{|x_1|^2} -\frac{x_1\cdot\xi_1}{|x_1|^4}x_1x_1^\top.
\end{equation}
With $Z_\varepsilon=X+i\varepsilon I_d$,
\begin{equation}
  g_{Z_\varepsilon}(z_1)=\varepsilon|x_1|^2\longrightarrow0.
\end{equation}
If $x_1=0$, take $Z_R=iRI_d$; then $g_{Z_R}(0,\xi_1)=R^{-1}|\xi_1|^2\to0$ as $R\to\infty$.
This concludes the proof.

\end{proof}

We are now ready to prove Theorem \ref{thm:necessity}.
\subsection{The proof of Theorem \ref{thm:necessity}}\label{sec:necessity}
Fix an arbitrary $Z\in\HH_d$.
By Proposition \ref{prop:Gaussian-determination},
\begin{equation}
  T\Phi_Z=c_Z\Phi_Q, \qquad Q=\mathcal{M}_S(Z)\in\HH_d,\; c_Z\neq0.
\end{equation}
By Proposition \ref{prop:complex-covariance-gaussian}, for every $\zeta\in\bC^{2d}$ and every real $r$,
\begin{equation}
  T\rho(r\zeta;0)\Phi_Z=c_Z\rho(rS\zeta;0)\Phi_Q.
\end{equation}
Since $T$ is bounded, by taking the norms we get:
\begin{equation}\label{eq:bounded-growth-ineq}
  |c_Z|^2\|\rho(rS\zeta;0)\Phi_Q\|_2^2 \leq \|T\|^2\|\rho(r\zeta;0)\Phi_Z\|_2^2.
\end{equation}
Then Lemma \ref{lem:norm-formula} yields,
\begin{equation}
 |c_Z|^2\|\Phi_Q\|_2^2e^{2\pi r^2F_Q(S\zeta)}\leq\|T\|^2\|\Phi_Z\|_2^2e^{2\pi r^2F_Z(\zeta)}.
\end{equation}
Take logarithms, divide by $2\pi r^2$, and let $|r|\to\infty$.
The constants disappear and we obtain
\begin{equation}\label{eq:F-contraction}
  F_Q(S\zeta)\leq F_Z(\zeta), \qquad \zeta\in\bC^{2d},\ Z\in\HH_d, \quad Q=\mathcal{M}_S(Z).
\end{equation}
Now, fix an arbitrary $\zeta\in\bC^{2d}$, decompose
\begin{equation}
  \zeta=z_0+iz_1, \qquad z_0,z_1\in\R^{2d}
\end{equation}
and write
\begin{equation}
  S\zeta=w_0+iw_1, \qquad w_0,w_1\in\R^{2d}.
\end{equation}
Since
\begin{equation}\label{eq:sigma-zeta-bar}
  \sigma(\zeta,\bar\zeta)=-2i\sigma(z_0,z_1),
\end{equation}
we have
\begin{equation}\label{eq:H-S}
  H_S(\zeta) :=i^{-1}\Big( \sigma(S\zeta,\overline{S\zeta}) -\sigma(\zeta,\bar\zeta) \Big) =2\big(\sigma(z_0,z_1)-\sigma(w_0,w_1)\big).
\end{equation}
On the other hand, \eqref{eq:F-contraction} and the definition of $F$ yield
\begin{align}
  0 &\leq F_Z(\zeta)-F_Q(S\zeta) =g_Z(z_1)-g_Q(w_1) +\sigma(z_0,z_1)-\sigma(w_0,w_1) =g_Z(z_1)-g_Q(w_1)+\frac12H_S(\zeta). \label{eq:key-H-ineq}
\end{align}
Since $Q\in\HH_d$, $g_Q$ is nonnegative.
Hence
\begin{equation}\label{eq:H-lower-gZ}
  H_S(\zeta)\geq -2g_Z(z_1) \qquad\text{for every }Z\in\HH_d.
\end{equation}
Crucially, the left-hand side is independent of $Z$.
By Lemma \ref{lem:squeeze}, choose $Z_n\in\HH_d$ with $g_{Z_n}(z_1)\to0$.
Passing to the limit in \eqref{eq:H-lower-gZ} gives
\begin{equation}
  H_S(\zeta)\geq0.
\end{equation}
Since $\zeta\in\bC^{2d}$ was arbitrary, this is precisely the condition $H_S\geq0$ of Definition \ref{intro.defPos}.
Therefore $S\in\Sp_+(d,\bC)$.

\begin{remark}\label{rem:minimal-covariance}
The proof of Theorem \ref{thm:necessity} does not use \eqref{eq:full-covariance} for arbitrary $f\in L^2$.
It is enough to assume that, for every $Z\in\HH_d$,
\begin{equation}
  T\rho(z;0)\Phi_Z=\rho(Sz;0)T\Phi_Z, \qquad z\in\R^{2d},
\end{equation}
holds in $\Shilovp$.
Thus positivity is already forced by the existence of a bounded operator that intertwines the complexified Schr\"odinger representation $\rho(z;\tau)$ with $z\in\rdd,\tau\in\bR$ only on Gaussian states.
\end{remark}

\begin{remark}\label{rem:all-gaussians}
For the final limit step of the proof, it is important to consider the full family $Z=X+iY\in\HH_d$, including chirped Gaussians with $X\neq0$.
Restricting only to real Gaussians, i.e., $Z=iM$ with $M>0$ does not, in general, allow one to make $g_Z(z_1)$ arbitrarily small for a prescribed phase-space vector $z_1=(x_1,\xi_1)$ with both components nonzero.
Indeed, the real part $X$ is what permits the cancellation $Xx_1=\xi_1$ in the proof of Lemma \ref{lem:squeeze}.
\end{remark}

\section{Characterization and uniqueness}\label{sec:characterization}

We now prove the converse and, at the same time, upgrade the dense intertwining relation of Proposition \ref{intro.propGRT} to all of $L^2$.

\begin{theorem}\label{thm:full-L2-positive}
Let $\widehat S\in\Mp_+(d,\C)$ have projection $S$.
Then, for every $f\in L^2(\R^d)$, $z\in\R^{2d}$ and $\tau\in\R$,
\begin{equation}\label{eq:full-L2-positive}
  \widehat S\rho(z;\tau)f=\rho(Sz;\tau)\widehat Sf \qquad\text{in }\Shilovp.
\end{equation}
Moreover, the right-hand side is represented by the $L^2$ function on the left.
Equivalently,
\begin{equation}\label{eq:range-domain}
  \mathrm{range}(\widehat S) \subseteq \{u\in L^2:\rho(Sz;\tau)u\text{ is represented by an }L^2\text{ function}\},
\end{equation}
and on all of $L^2$,
\begin{equation}
  \rho(Sz;\tau)\widehat S=\widehat S\rho(z;\tau)\in\cB(L^2(\R^d)).
\end{equation}
\end{theorem}

\begin{proof}
The assertion is known to hold densely on $L^2(\rd)$, as stated in Proposition \ref{intro.propGRT}.

Fix $z\in\R^{2d}$ and $\tau\in\R$.
Let $f_n$ belong to that dense subspace and $f_n\to f$ in $L^2$.
Since $\widehat S$ and $\rho(z;\tau)$, $z\in\rdd$, $\tau\in\bR$, are bounded on $L^2$,
\begin{equation}
  \widehat S\rho(z;\tau)f_n \longrightarrow \widehat S\rho(z;\tau)f \quad\text{in }L^2\subset\Shilovp.
\end{equation}
Also,
\begin{equation}
  \widehat Sf_n\longrightarrow\widehat Sf \quad\text{in }L^2\subset\Shilovp.
\end{equation}
For fixed complex phase point $Sz$, the operator $\rho(Sz;\tau)$ is weakly continuous on $\Shilovp$ by its transpose construction in \eqref{eq:M-dual}-\eqref{eq:rho-Gprime}.
Therefore
\begin{equation}
  \rho(Sz;\tau)\widehat Sf_n \longrightarrow \rho(Sz;\tau)\widehat Sf \quad\text{weakly in }\Shilovp.
\end{equation}
Passing to the limit in the dense intertwining relation gives \eqref{eq:full-L2-positive}.
The left-hand side belongs to $L^2$, so the right-hand side is represented by the same $L^2$ function.
This proves \eqref{eq:range-domain} as well.

\end{proof}

It remains to prove that the intertwiner is unique up to a constant.
Observe that in the next result, no a priori symplectic projection of either $T_1$ or $T_2$ is used.

\begin{theorem}\label{thm:uniqueness}
Let $S\in\Sp(d,\C)$ and suppose that $0\neq T_1,T_2\in\cB(L^2(\R^d))$ both satisfy \eqref{eq:full-covariance} with the same matrix $S$.
Then there exists $c\in\C\setminus\{0\}$ such that
\begin{equation}\label{eq:T1-cT2}
  T_1=c\cdot T_2 \qquad\text{on }L^2(\R^d).
\end{equation}
\end{theorem}

\begin{proof}
Choose any $Z\in\HH_d$, for instance $Z=iI_d$.
By Proposition \ref{prop:Gaussian-determination},
\begin{equation}
  T_j\Phi_Z=c_j\cdot \Phi_{\cM_S(Z)}, \qquad c_j\neq0,\; j=1,2.
\end{equation}
Thus, with $c=c_1/c_2$, $T_1\Phi_Z=c\cdot T_2\Phi_Z$.
For every real $z\in\R^{2d}$,
\begin{align}
  T_1\rho(z;0)\Phi_Z=\rho(Sz;0)T_1\Phi_Z=c\cdot\rho(Sz;0)T_2\Phi_Z=c\cdot T_2\rho(z;0)\Phi_Z.
\end{align}
The space $\overline{\mathrm{span}\{\rho(z;0)\Phi_Z:z\in\R^{2d}\}}$ is a nonzero closed invariant subspace for the irreducible Schr\"odinger representation $\rho$, hence it is all of $L^2$.
Since $T_1$ and $T_2$ are bounded, \eqref{eq:T1-cT2} follows.
\end{proof}

\begin{remark}\label{rem:minimal-uniqueness}
For Theorem \ref{thm:uniqueness}, the validity of the intertwining relation on $L^2(\rd)$ is again stronger than necessary.
It is enough that the two operators satisfy the same intertwining relation on the $\mathrm{span}\{\rho(z;0)\Phi_Z\}_{z\in\rdd}$ of one fixed nonzero Siegel Gaussian $\Phi_Z$.
\end{remark}

As a corollary, we identify the intertwiners of the complexified Schr\"odinger representation with the operators in the oscillator semigroup.

\begin{corollary}\label{cor:canonical-identification}
Let $S\in\Sp(d,\C)$ and let $T$ be a bounded Schr\"odinger intertwiner of $S$.
Then $S\in\Sp_+(d,\C)$ and, for every oscillator-semigroup lift $\widehat S$ of $S$, it is $T=c\cdot \widehat S$ for some $c\in\C\setminus\{0\}$.
\end{corollary}

\begin{proof}
By Theorem \ref{thm:necessity}, $S$ is positive.
By Theorem \ref{thm:full-L2-positive}, $\widehat S$ is itself a bounded intertwiner of $S$.
Apply Theorem \ref{thm:uniqueness}.
\end{proof}

Finally, we collect the results shown so far into the proof of Theorem \ref{thm:main}.

\subsection{Proof of Theorem \ref{thm:main}}
The implication $(ii)\Rightarrow(i)$ is Theorem \ref{thm:necessity}.
Conversely, if $S\in\Sp_+(d,\C)$, take $\widehat S\in\Mp_+(d,\bC)$ so that $\pi^{Mp}_+(\widehat S)=S$.
Theorem \ref{thm:full-L2-positive} shows that it satisfies the full intertwining relation \eqref{eq:full-covariance}.
The one-dimensionality statement follows from Theorem \ref{thm:uniqueness}, while the final part of the statement is Corollary \ref{cor:canonical-identification}.

\begin{appendix}
\section{Proof of Lemma \ref{lem:exp-multiplier}}\label{appendix:A}

By the Gaussian characterization of $\Shilov$, see e.g.\ \cite[Proposition~2.2(a)]{Petersson}, for every $\varphi\in\Shilov$ there exist $C,A,a>0$ such that
\begin{equation}\label{eq:appendix-gaussian}
    |\partial^\mu\varphi(t)| \leq C A^{|\mu|}(\mu!)^{1/2}e^{-a|t|^2}, \qquad \mu\in\N^d,\ t\in\R^d.
\end{equation}
Fix $\zeta_0\in\bC$ and set $\psi(t)=e^{\zeta_0\lambda\cdot t}\varphi(t)$.
Since, by Young's inequality,
\begin{equation}
    e^{\Re(\zeta_0\lambda)\cdot t} \leq e^{|\zeta_0||\lambda||t|} \leq C_0 e^{a|t|^2/2},
\end{equation}
Leibniz' formula and \eqref{eq:appendix-gaussian} give, after possibly changing the constants,
\begin{equation}\label{eq:psi-gaussian-estimate}
    |\partial^\mu\psi(t)| \leq C_1 A_1^{|\mu|}(\mu!)^{1/2}e^{-a_1|t|^2}, \qquad \mu\in\N^d,
\end{equation}
for some $C_1,A_1,a_1>0$.
Indeed,
\begin{equation}
\begin{split}
|\partial^\mu\psi(t)| &= \Big| e^{\zeta_0\lambda\cdot t} \sum_{\kappa\leq\mu} \binom{\mu}{\kappa} (\zeta_0\lambda)^\kappa \partial^{\mu-\kappa}\varphi(t) \Big| \leq e^{\Re(\zeta_0\lambda)\cdot t} \sum_{\kappa\leq\mu} \binom{\mu}{\kappa} |(\zeta_0\lambda)^\kappa|\cdot|\partial^{\mu-\kappa}\varphi(t)|.
\end{split}
\end{equation}
Setting $b:=|\zeta_0|\cdot|\lambda|$, we obtain
\begin{equation}
\begin{split}
|\partial^\mu\psi(t)| &\leq C e^{b|t|-a|t|^2} \sum_{\kappa\leq\mu} \binom{\mu}{\kappa} b^{|\kappa|} A^{|\mu-\kappa|} ((\mu-\kappa)!)^{1/2} \leq C e^{b|t|-a|t|^2} (\mu!)^{1/2} \sum_{\kappa\leq\mu} \binom{\mu}{\kappa} b^{|\kappa|} A^{|\mu|-|\kappa|}.
\end{split}
\end{equation}
By the multi-index binomial formula,
\begin{equation}
\sum_{\kappa\leq\mu} \binom{\mu}{\kappa} b^{|\kappa|} A^{|\mu|-|\kappa|} = (A+b)^{|\mu|}.
\end{equation}
Moreover, $b|t|-\frac a2|t|^2 \leq \frac{b^2}{2a}$, and therefore $e^{b|t|-a|t|^2} \leq e^{b^2/(2a)}e^{-a|t|^2/2}$.
Hence
\begin{equation}
|\partial^\mu\psi(t)| \leq C e^{b^2/(2a)} (A+b)^{|\mu|} (\mu!)^{1/2} e^{-a|t|^2/2}.
\end{equation}
Thus \eqref{eq:psi-gaussian-estimate} holds with
\begin{equation}
C_1=C e^{b^2/(2a)}, \qquad A_1=A+b, \qquad a_1=\frac a2.
\end{equation}
Consider now the $n$th Taylor coefficient in \eqref{eq:exp-series},
\begin{equation}
    f_n(t)=\frac{(\lambda\cdot t)^n}{n!}\psi(t).
\end{equation}
For $\nu\leq\gamma$, setting $k=|\nu|$, we have
\begin{equation}
    \partial^\nu(\lambda\cdot t)^n = \frac{n!}{(n-k)!}\lambda^\nu (\lambda\cdot t)^{n-k}, \qquad k\leq n,
\end{equation}
whereas the derivative vanishes for $k>n$.
Hence, by \eqref{eq:psi-gaussian-estimate},
\begin{equation}
\begin{split}
    |t^\beta\partial^\gamma f_n(t)| &\leq C_1 \sum_{\substack{\nu\leq\gamma\\|\nu|\leq n}} \binom{\gamma}{\nu} \frac{|\lambda|^{|\nu|}}{(n-|\nu|)!} |t^\beta|\cdot|\lambda\cdot t|^{n-|\nu|} A_1^{|\gamma-\nu|} ((\gamma-\nu)!)^{1/2} e^{-a_1|t|^2}.
\end{split}
\end{equation}
We now split the Gaussian factor as
\begin{equation}
    e^{-a_1|t|^2} = e^{-a_1|t|^2/2}e^{-a_1|t|^2/2}.
\end{equation}
For $n\in\N$, the function
\begin{equation}
    r\longmapsto r^n e^{-a_1r^2/2}, \qquad r\geq0,
\end{equation}
attains its maximum at $r=(n/a_1)^{1/2}$, and therefore
\begin{equation}
    \sup_{r\geq0}r^n e^{-a_1r^2/2} = \left(\frac{n}{a_1e}\right)^{n/2}.
\end{equation}
Since $n!\geq \left( n/e\right)^n$, we obtain
\begin{equation}
    \sup_{r\geq0}r^n e^{-a_1r^2/2} \leq a_1^{-n/2}(n!)^{1/2}.
\end{equation}
Hence, using
\begin{equation}
    |t^\beta|e^{-a_1|t|^2/2} = \prod_{j=1}^d |t_j|^{\beta_j}e^{-a_1t_j^2/2},
\end{equation}
we find
\begin{equation}
\begin{split}
    \sup_{t\in\R^d} |t^\beta|e^{-a_1|t|^2/2} &= \prod_{j=1}^d \sup_{t_j\in\R} |t_j|^{\beta_j}e^{-a_1t_j^2/2}\leq a_1^{-|\beta|/2}(\beta!)^{1/2}.
\end{split}
\end{equation}
Thus
\begin{equation}
    \sup_{t\in\R^d} |t^\beta|e^{-a_1|t|^2/2} \leq B_0^{|\beta|}(\beta!)^{1/2}, \qquad B_0:=\max\{1,a_1^{-1/2}\}.
\end{equation}
Similarly, since the expression depends only on $|t|$,
\begin{equation}
\begin{split}
    \sup_{t\in\R^d}|t|^m e^{-a_1|t|^2/2} &= \sup_{r\geq0}r^m e^{-a_1r^2/2}= \left(\frac{m}{a_1e}\right)^{m/2} \leq a_1^{-m/2}(m!)^{1/2}.
\end{split}
\end{equation}
Consequently,
\begin{equation}
    \sup_{t\in\R^d}|t|^m e^{-a_1|t|^2/2} \leq B_1^m(m!)^{1/2}, \qquad B_1:=\max\{1,a_1^{-1/2}\}.
\end{equation}
Therefore, writing $m=n-k$ yields
\begin{equation}
\begin{split}
    \frac{1}{m!} \sup_{t\in\R^d} |\lambda\cdot t|^m e^{-a_1|t|^2/2} &\leq \frac{|\lambda|^m}{m!} \sup_{t\in\R^d} |t|^m e^{-a_1|t|^2/2} \leq \frac{(B_1|\lambda|)^m}{(m!)^{1/2}}.
\end{split}
\end{equation}
Moreover,
\begin{equation}
    \frac1{\sqrt{(n-k)!}} = \frac1{\sqrt{n!}} \left(\frac{n!}{(n-k)!}\right)^{1/2} \leq \frac{2^{n/2}\sqrt{k!}}{\sqrt{n!}}.
\end{equation}
Since $k=|\nu|$, the multinomial formula gives $k!\leq d^k\nu!$, and therefore
\begin{equation}
\begin{split}
    \binom{\gamma}{\nu} (\nu!)^{1/2}((\gamma-\nu)!)^{1/2} &= (\gamma!)^{1/2} \binom{\gamma}{\nu}^{1/2}\leq (\gamma!)^{1/2} \binom{\gamma}{\nu}.
\end{split}
\end{equation}
Summing over $\nu\leq\gamma$, all factors depending on $\nu$ can thus be absorbed into an exponential in $|\gamma|$.
We obtain constants $C_2,B_2,B_3,C_3>0$, independent of $n,\beta,\gamma$, such that
\begin{equation}\label{eq:fn-estimate}
    \sup_{t\in\R^d} |t^\beta\partial^\gamma f_n(t)| \leq C_2 B_2^{|\beta|} B_3^{|\gamma|} (\beta!\gamma!)^{1/2} \frac{C_3^n}{(n!)^{1/2}}.
\end{equation}
Choosing
\begin{equation}
    h\geq\max\{B_2,B_3\},
\end{equation}
it follows that
\begin{equation}
    \|f_n\|_h \leq C_2\frac{C_3^n}{(n!)^{1/2}}.
\end{equation}
Therefore, uniformly for $|\zeta-\zeta_0|\leq R$,
\begin{equation}
    \left\| \frac{(\zeta-\zeta_0)^n}{n!} (\lambda\cdot t)^n e^{\zeta_0\lambda\cdot t}\varphi(t) \right\|_h \leq C_2\frac{(RC_3)^n}{(n!)^{1/2}}.
\end{equation}
Therefore, the Taylor series \eqref{eq:exp-series} converges absolutely and uniformly on compact subsets of $\bC$, with values in the fixed Banach step $\Shilovh$.
In particular, $ \zeta\longmapsto e^{\zeta\lambda\cdot t}\varphi(t)$ is an entire $S_{1/2}^{1/2}$-valued function.
The same argument applies to the termwise differentiated series.
Indeed, on $|\zeta-\zeta_0|\leq R$ its terms are bounded, up to a constant depending on $R$, by
\begin{equation}
    n\frac{(RC_3)^n}{(n!)^{1/2}},
\end{equation}
and the corresponding series is again summable.
Hence differentiation may be performed termwise in $\Shilovh$, which yields \eqref{eq:exp-derivative}.
Finally, the same estimates, with $\lambda$ replaced by $-\lambda$, show that multiplication by $e^{-\zeta\lambda\cdot t}$ acts on $\Shilov$ as well.
Since $e^{-\zeta\lambda\cdot t}e^{\zeta\lambda\cdot t}=1$, the two multipliers are inverse to each other.
This concludes the proof.

\end{appendix}

\end{document}